\documentclass[letterpaper]{amsart}

\usepackage[T1]{fontenc}
\usepackage{lmodern}
\usepackage{diagrams}

\usepackage[pdfusetitle]{hyperref}

\def\Z{\mathbb Z}
\def\Q{\mathbb Q}
\def\R{\mathbb R}
\def\C{\mathbb C}
\def\F{\mathbb F}

\let\phi\varphi
\def\t{\mathfrak t}
\def\O{\mathcal O}
\def\CP{\mathbb{CP}}
\DeclareMathOperator{\lin}{lin}
\DeclareMathOperator{\Hom}{Hom}

\DeclareMathOperator{\lk}{lk}
\def\at#1{{\left.#1\right|}}
\def\pair#1{{\langle #1\rangle}}

\def\FF{\mathcal F}
\def\Sigmamax{\Sigma_{\rm max}}
\def\TC{T}
\def\TA{\mathbb T}
\def\PP{P\mkern -1mu P}
\def\Hodd{H^{\rm odd}}

\def\cf{\emph{cf.}}

\def\ie{\emph{i.e.}}

\def\arxiv#1{\href{http://arxiv.org/abs/#1}{\texttt{arXiv:#1}}}
\def\doi#1{\href{http://dx.doi.org/#1}{doi:#1}}

\let\osubsection\subsection
\def\subsection#1{\osubsection{\boldmath{#1}}}

\theoremstyle{plain}
\newtheorem{theorem}{Theorem}[section]
\newtheorem{proposition}[theorem]{Proposition}
\newtheorem{lemma}[theorem]{Lemma}
\newtheorem{corollary}[theorem]{Corollary}

\theoremstyle{definition}

\newtheorem{question}[theorem]{Question}

\theoremstyle{remark}

\newtheorem*{acknowledgments}{Acknowledgments}

\numberwithin{equation}{section}

\begin{document}

\title[Describing toric varieties and their equivariant cohomology]%
  {Describing toric varieties and\\their equivariant cohomology}
\author{Matthias Franz}
\thanks{The author was supported by an NSERC Discovery Grant.}
\address{Department of Mathematics, University of Western Ontario,
  London, Ont.\ N6A\;5B7, Canada}
\email{mfranz@uwo.ca}

\begin{abstract}
  Topologically, compact toric varieties can be constructed
  as identification spaces: they are quotients of the
  product of a compact torus and the order complex of the fan.
  We give a detailed proof of this fact, extend it to the non-compact case
  and draw several, mostly cohomological conclusions.

  In particular, we show
  that the equivariant integral cohomology of a toric variety can be
  described in terms of piecewise polynomials on the fan
  if the ordinary integral cohomology is concentrated in even degrees.
  This generalizes a result of Bahri--Franz--Ray to the non-compact case.
  We also investigate torsion phenomena in integral cohomology.
\end{abstract}

\maketitle

\section{Introduction}

\noindent
Let $\TC\cong(S^1)^n$ be a torus with Lie algebra~$\t\cong\R^n$,
and $P$ a full-dimensional polytope in the dual~$\t^*$ of~$\t$,
integral % or rational?
with respect to the weight lattice. %~$M\subset\t^*$.
The toric variety~$X_P$ associated with~$P$ is projective,
and $\TC$ as well as its complexification~$\TA\cong(\C^*)^n$ act on it.

In the article~\cite{Jurkiewicz:81},
Jurkiewicz showed that one can recover $P$
as the image of the moment map~$X_P\to\t^*$,
and that this map is the quotient of~$X_P$ by the action
of the compact torus~$\TC$. (See also Atiyah~\cite{Atiyah:82}.)
Since $X_P/\TC$ is canonically embedded into~$X_P$
as its non-negative part, %~$X_P^+$,
% (\cf~\cite[Prop.~1.8]{Oda:88} or~\cite[Sec.~4.1]{Fulton:93}),
Jurkiewicz's result immediately yields a
$\TC$-equivariant homeomorphism
\begin{equation}
  X_P\approx\bigl(\TC\times P\bigr)\bigm/\mathord\sim\,,
\end{equation}
where two points $(t_1,x_1)$,~$(t_2,x_2)\in \TC\times P$
are identified if $x_1=x_2$, say
with supporting face~$f$ of~$P$,
% contained in the relative interior of the face~$f$ of~$P$,
and if $t_1 t_2^{-1}$ lies in the subtorus % ~$\TC_f$
of~$\TC$
whose Lie algebra is the annihilator of~$\lin(f-x_1)$.
% This follows readily from two facts:
% \begin{enumerate}
% \item $P\approx X_P/\TC$.
%   (For example, consider an embedding~$P\subset\CP^m$
%   and the moment map~$\mu\colon\CP^m\to\R^m$.
%   Then $P$ can be identified with the image~$\mu(X_P)$,
%   and the fibres of~$\mu\colon X_P\to P$ are the $\TC$-orbits.)
%   \comment{Atiyah}
% \item The quotient~$X_P/\TC$ is homeomorphic to
%   the non-negative part~$X_P^+$ of~$X_P$.
%   \comment{ref Oda}
% \end{enumerate}
The aim of this note is to prove a similar description
for arbitrary toric varieties.

Let $\Sigma$ be a not necessarily complete fan in~$\t$,
rational with respect to the
lattice %~$N\subset\t$
of $1$-parameter subgroups of~$\TC$,
% (identified with their differentials),
and let $\FF(\Sigma)$ be the order complex of~$\Sigma$.
% Its topological realization~$|\FF(\Sigma)|$ is an $n$-ball
% if $\Sigma$ is complete.
(If $\Sigma$ is the normal fan of the polytope~$P$, then
$\FF(\Sigma)$ can be thought of as the barycentric subdivision of~$P$.)
Define the $\TC$-space
\begin{equation}\label{topology-complete-C}
  Y_\Sigma=
    \bigl(\TC\times |\FF(\Sigma)|\bigr)\bigm/\mathord\sim\,,
\end{equation}
with the following identification:
For~$x\in|\FF(\Sigma)|$,
say with supporting simplex~$\alpha=(\sigma_0,\dots,\sigma_p)$,
one has $(t_1,x)\sim(t_2,x)$
if $t_1 t_2^{-1}$ % \in \TC_{\sigma_0}$, where $\TC_{\sigma_0}$ denotes
lies in the subtorus~$\TC_{\sigma_0}$ of~$\TC$
whose Lie algebra is the linear span of~$\sigma_0$.

This construction has appeared in Davis--Januszkiewicz's work
on (quasi)toric manifolds \cite{DavisJanuszkiewicz:91},
but without linking it with algebraic geometry.
Fischli and Yavin \cite{Fischli:92}, \cite{FischliYavin:94}, \cite{Yavin:91}
attributed construction~\eqref{topology-complete-C} to MacPherson
and used it as the \emph{definition} of a toric variety.
Because since then several authors
\cite{Jordan:97}, \cite{WelkerZieglerZivaljevic:99},
\cite{Civan:02}, \cite{Panov:06}
%
% just quoted:
% BuchstaberPanov:02, p. 83
% Buchstaber, Ray: Invitation to toric topology
% Feichtner: Rational versus real cohomology algebras of low-dimensional toric varieties
% Karshon: Periodic Hamiltonian flows on four dimensional manifolds
% Masuda, Suh: Classification problems of toric manifolds via topology
%
have applied \eqref{topology-complete-C}
to toric varieties in the usual sense,
we feel that it might be beneficial
to supply a justification for this.

\begin{theorem}\label{main-result-C}
  % Then $|\FF(\Sigma)|\approx D^n$ and
  % Let $k=\C$,~$\R$ or~$\R_+$.
  If $\Sigma$ is complete, then $Y_\Sigma$ is
  $\TC$-equivariantly homeomorphic
  to~$X_\Sigma$.
  In general, $Y_\Sigma$ is a
  $\TC$-equivariant strong deformation retract of~$X_\Sigma$.
  % The homotopy~$X_\Sigma(k)\times[0,1]\to Y_\Sigma(k)$
  % can be chosen to preserve orbit types.
\end{theorem}

A basic tool to study transformation groups are equivariant CW~complexes.
% (\cf~\cite[Sec.~1.1]{AlldayPuppe:93} for a definition).
Since $Y_\Sigma$ is a finite $\TC$-CW complex by construction
(with $\TC$-cells~$(\TC/\TC_{\sigma_0})\times|\alpha|$
in the notation used above), we can immediately conclude:

\begin{corollary}\label{T-CW-complex-C}
  If $\Sigma$ is complete, then the toric variety~$X_\Sigma$
  is a (necessarily finite) $\TC$-CW complex. In general,
  $X_\Sigma$ has the equivariant homotopy type of a
  finite $\TC$-CW complex.
\end{corollary}

These results do not only apply to complex, but also to real
toric varieties and to their non-negative parts. 
In Section~\ref{toric-varieties-general} we introduce the
notation necessary to formulate these generalizations;
the proof of the topological description then appears
in Section~\ref{proof-main-result}.
Some remarks about cubical subdivisions are made
in Section~\ref{cubical}.

\medskip

The second part of this paper studies 
the ordinary and equivariant singular cohomology
of toric varieties.
We use Corollary~\ref{T-CW-complex-C} to generalize a result
of Bahri--Franz--Ray \cite[Prop.~2.2]{BahriFranzRay:09} from the projective
to the general case, in particular to non-compact toric varieties.
This was announced in~\cite[Remark~2]{BahriFranzRay:09}.

\begin{theorem}\label{no-odd-degree}
  Let $X_\Sigma$ be a toric variety.
  If $H^*(X_\Sigma;\Z)$ vanishes in odd degrees,
  then
  $H_\TC^*(X_\Sigma;\Z)$
  % the $\TC$-equivariant cohomology of~$X_\Sigma$
  is isomorphic,
  as algebra over the polynomial ring~$H^*(B\TC;\Z)$,
  to $\PP(\Sigma;\Z)$,
  the ring of integral piecewise polynomials on~$\Sigma$.
\end{theorem}

The proof appears in Section~\ref{proof-no-odd-degree},
together with the precise definition of piecewise polynomials.
Combining Theorem~\ref{no-odd-degree} with a result of Payne~\cite{Payne:06},
we get:

\begin{corollary}
  Let $X_\Sigma$ be a toric variety.
  If $H^*(X_\Sigma;\Z)$ vanishes in odd degrees,
  then $H_T^*(X_\Sigma;\Z)$ is isomorphic to~$A_T^*(X_\Sigma)$,
  the equivariant Chow cohomology ring of~$X_\Sigma$.
\end{corollary}

We finally investigate torsion phenomena
in the integral cohomology. The celebrated Jurkiewicz--Danilov theorem
implies that no torsion appears if $X_\Sigma$ is smooth
and compact. In Section~\ref{torsion-free} this is extended as follows:

\begin{proposition}\label{no-torsion-smooth-compact}
  Assume that % the toric variety~
  $X_\Sigma$ is smooth or compact.
  If $H^*(X_\Sigma;\Z)$ is concentrated in even degrees,
  then it is torsion-free.
  % Moreover, $H_\TC^*(X_\Sigma;\Z)$ is free over~$H^*(B\TC;\Z)$ in this case.
\end{proposition}

\begin{acknowledgments}
  This paper is a by-product of a joint project with Tony Bahri and Nigel Ray
  on toric orbifolds (see \cite{BahriFranzRay:09}, \cite{BahriFranzRay:three},
  \cite{BahriFranzRay:two}); it is a pleasure to acknowledge this
  fruitful collaboration.
  I would also like to thank Sam Payne for stimulating discussions,
  and again Nigel Ray for spotting a mistake in an earlier version
  of this paper and for other valuable comments.
\end{acknowledgments}

\section{Toric varieties defined over monoids}
\label{toric-varieties-general}

\noindent
In this section we brief\/ly recall how to define
toric varieties over submonoids of~$\C$, and then
state analogues of Theorem~\ref{main-result-C} and
Corollary~\ref{T-CW-complex-C} %~and~\ref{diagram-C}.
Standard references for toric varieties
are \cite{Oda:88}, \cite{Fulton:93} and~\cite{Ewald:96};
see in particular \cite[Sec.~1.3]{Oda:88} and~\cite[Sec.~4.1]{Fulton:93}
for real toric varieties and non-negative parts.

Let $N$ be a free $\Z$-module of rank~$n$ with dual~$M=N^\vee$.
Extensions to real scalars are written in the form~$N_\R=N\otimes\R$.
(Unless stated otherwise, tensor products are taken over~$\Z$.)

Let $k$ be a multiplicative submonoid of~$\C$ containing $0$~and~$1$.
We write $\TA(k)=\Hom(M,k)$
for the group of monoid homomorphisms~$M\to k$
(or, in this case equivalently, of group homomorphisms~$M\to k^\times$)
% We identify $\TA(\C)$ with~$\TA$
and set $\TC(k)=\TA(k\cap S^1)$.
Then $\TA(\C)$~and~$\TC(\C)$
are the algebraic torus~$\TA$ and the compact torus~$\TC$
introduced earlier, and
$\TC(\R)\cong(\Z_2)^n$ is the compact form of~$\TA(\R)\cong(\R^*)^n$.
Taking~$k=\R_+=[0,\infty)$,
we get $\TA(\R_+)\cong(0,\infty)^n$ and $\TC(\R_+)=1$.

For a rational cone~$\sigma\subset N_\R$ % =N\otimes_\Z\R$
with dual~$\sigma^\vee\subset M_\R$,
the affine toric variety~$X_\sigma(k)$ %, $\sigma\in\Sigma$
is defined as the set of monoid homomorphisms
\begin{equation}
  X_\sigma(k) = \Hom(\sigma^\vee\cap M,k).
\end{equation}
Since $\sigma^\vee\cap M$ is finitely generated, we may embed
$\Hom(\sigma^\vee\cap M,k)$ into some affine space~$k^L$, which induces a
topology on~$X_\sigma(k)$. (We always use the metric topology
on~$\C$, hence on~$X_\sigma(k)$.)
The ``tori''~$\TA(k)$~and~$\TC(k)$
act on~$X_\sigma(k)$ by pointwise multiplication of functions.

As in the introduction, $\Sigma$ denotes a rational fan
in~$N_\R$. %=N\otimes_\Z\R$.
The toric variety~$X_\Sigma(k)$ is obtained by gluing the
affine pieces~$X_\sigma(k)$ together as prescribed by the fan;
it is a Hausdorff space. % , and compact if $\Sigma$ is complete.
We write $x_\sigma\in X_\sigma(k)$ for the distinguished point
of~$X_\sigma(k)$,
\begin{equation}
  x_\sigma(m)=\begin{cases}
    1 & \hbox{if $m\in\sigma^\perp$} \\
    0 & \hbox{otherwise},
  \end{cases}
\end{equation}
and $\O_\sigma(k)$ for its orbit under~$\TA(k)$.
% (\cf~\cite[Sec.~2.1]{Fulton:93}.
% (The point~$x_\sigma$ is independent of~$k$ in the sense
% that it lies in~$X_\sigma(\{0,1\})\subset X_\sigma(k)$.)
These orbits partition $X_\Sigma(k)$.

Note that $X_\Sigma(\C)$ is the usual complex toric variety,
$X_\Sigma(\R)$ its real part and $X_\Sigma(\R_+)$ its
non-negative part (the `associated manifold with corners').
%, \cf~\cite[Sec.~4.1]{Fulton:93}.
As done in the introduction, we often write
% $\TA=\TA(\C)$, $\TC=\TC(\C)$,
$X_\Sigma=X_\Sigma(\C)$ and $\O_\sigma=\O_\sigma(\C)$.
We also use the notation~%
% $\TC_\sigma(k)=\{g\in T(k): g(\sigma\cap N)=1\}$
$\TC_\sigma(k)=\{\,g\in T(k): \sigma^\perp\cap M\subset\ker g\,\}$
for the subgroup of~$T(k)$ determined by~$\sigma\in\Sigma$.

Recall that the $p$-simplices of the order complex~$\FF(\Sigma)$ % of~$\Sigma$
are the strictly ascending
sequences~$\sigma_0<\dots<\sigma_p$ of length~$p+1$
in the partially ordered set~$\Sigma$.
(This is the same as the nerve of~$\Sigma$, considered as a category
with order relations as morphisms.)
In particular, vertices of~$\FF(\Sigma)$ correspond to cones in~$\Sigma$.
One may think of $\FF(\Sigma)$ as the cone over the barycentric subdivision
of the ``polyhedral complex'' obtained by intersecting
the unit sphere~$S^{n-1}\subset N_\R$ with~$\Sigma$.

In analogy with~\eqref{topology-complete-C},
we define the $\TC(k)$-space
\begin{equation}\label{topology-complete}
  Y_\Sigma(k)=
    \bigl(\TC(k)\times |\FF(\Sigma)|\bigr)\bigm/\mathord\sim\,,
\end{equation}
where the identification is done as follows:
For~$x\in|\FF(\Sigma)|$, 
say with supporting simplex~$\alpha=(\sigma_0,\dots,\sigma_p)$,
one has
$(t_1,x)\sim(t_2,x)$ iff
$t_1 t_2^{-1}\in \TC_{\sigma_0}(k)$.

We have the following generalizations of results
stated in the introduction, where $k$ denotes either
$\C$,~$\R$ or~$\R_+$.
The proof of Theorem~\ref{main-result} appears in the following section.

\begin{theorem}\label{main-result}
  If $\Sigma$ is complete, then $Y_\Sigma(k)$ is
  $\TC(k)$-equivariantly homeomorphic
  to~$X_\Sigma(k)$.
  In general, $Y_\Sigma(k)$ is a
  $\TC(k)$-equivariant strong deformation retract of~$X_\Sigma(k)$.
  % The homotopy~$X_\Sigma(k)\times[0,1]\to Y_\Sigma(k)$
  % can be chosen to preserve orbit types.
\end{theorem}

Equivariant CW~complexes are defined
in~\cite[Sec.~1.1]{AlldayPuppe:93}, for instance.

\begin{corollary}\label{T-CW-complex}
  If $\Sigma$ is complete, then $X_\Sigma(k)$
  is a finite $\TC(k)$-CW complex. In general,
  $X_\Sigma(k)$ has the equivariant homotopy type of a
  finite $\TC(k)$-CW complex.
\end{corollary}

Let $D_\Sigma(k)$ be the diagram of spaces over~$\Sigma$ that
% (\cf~\cite[Sec.~2.1]{WelkerZieglerZivaljevic:99})
assigns $\TC(k)/\TC_\sigma(k)$ to~$\sigma\in\Sigma$,
and the projection~$\TC(k)/\TC_\sigma(k)\to \TC(k)/\TC_\tau(k)$
to~$\sigma\le\tau$.
Comparing \eqref{topology-complete}
with the standard construction of a homotopy colimit
of a diagram of spaces,
\cf~\cite[Sec.~2]{WelkerZieglerZivaljevic:99},
we arrive at the following observation, which was made
% The following result was stated
in~\cite[Prop.~5.3]{WelkerZieglerZivaljevic:99}
for compact complex toric varieties:

\begin{corollary}
  % Let $k=\C$,~$\R$ or~$\R_+$.
  The space~$X_\Sigma(k)$ is the homotopy colimit
  of the diagram~$D_\Sigma(k)$.
\end{corollary}

\section{A topological description of toric varieties}
\label{proof-main-result}

\noindent
We write $\Sigma_i\subset\Sigma$ for the subset of $i$-dimensional cones
% of~$\Sigma$
and $\Sigmamax$ for the set of (with respect to inclusion) maximal cones.
For a cone~$\sigma\in\Sigma$,
let $N_\sigma$ be the intersection of~$N$ with the linear hull of~$\sigma$,
and $\pi_\sigma\colon N\to N(\sigma)=N/N_\sigma$ be the quotient map
as well as its analogue over~$\R$.

\subsection{Case of
  complete~\texorpdfstring{$\Sigma$}{Sigma}
  and~\texorpdfstring{$k=\R_+$}{k=R+}}
\label{proof-complete-positive}

\noindent
Since $\TC(k)=1$ is trivial in this case,
it suffices to exhibit a triangulation
of the non-negative part~$X_\Sigma(\R_+)$
% of~$X_\Sigma$
isomorphic with~$\FF(\Sigma)$.
For each simplex~$\alpha=(\sigma_0,\dots,\sigma_p)\in\FF(\Sigma)$,
we will construct a % singular
$p$-simplex~$B(\alpha)$ in~$X_{\sigma_p}(\R_+)$
whose interior%
\footnote{By `interior' we always mean `relative interior'.}
lies in the orbit~$\O_{\sigma_0}(\R_+)$
corresponding to the initial vertex~$\sigma_0$ of~$\alpha$.
% This will be done inductively.

Choose a point~$v_\sigma\in N$ in the interior
of~$\sigma$, for example the sum of the
minimal integral generators of the extremal rays of~$\sigma$.
Let
\begin{equation}
  \lambda_\sigma\colon(0,\infty)\to \TA(\R_+),
  \quad
  t\mapsto\Bigl(m\mapsto t^\pair{m,v_\sigma}\Bigr) % \exp(t\,v_\sigma)
\end{equation}
be the corresponding $1$-parameter subgroup.
% of~$T$ with differential~$v_\sigma$.
Because $\Sigma$ is complete, the variety~$X_\Sigma(\R_+)$ is compact,
so that the limit~$\lambda_\sigma(0)\,x:=\lim_{t\to0}\lambda_\sigma(t)\,x$
exists for all~$x\in X_\Sigma(\R_+)$. For the limits we are interested in,
this will become evident during the proof of the following lemma.

\begin{lemma}\label{image-cube}
  Let $\alpha=(\sigma_0,\dots,\sigma_p)\in\FF(\Sigma)$ be a $p$-simplex,
  $p\ge1$.
  Then the $\TA(\R_+)$-action on~$X_\Sigma(\R_+)$ induces a continuous map
  \begin{equation*} %\label{definition-A}
    \phi_{\alpha}\colon[0,1]^p\to X_{\sigma_p}(\R_+),
    \quad
    t=(t_1,\dots,t_p)\mapsto
      \lambda_{\sigma_p}(t_p)\cdots\lambda_{\sigma_1}(t_1)\,x_{\sigma_0}.
  \end{equation*}
  Moreover, $\phi_{\alpha}(t)=x_{\sigma_p}$
  if $t_p=0$, and
  $\phi_{\alpha}(t)\neq \phi_{\alpha}(t')$
  if $t_p\neq t'_p$.
\end{lemma}

\begin{proof}
  By the definition of the action of~$\TA(k)=\Hom(M,k)$
  on~$X_\sigma(k) = \Hom(\sigma^\vee\cap M,k)$,
  we have
  \begin{equation*}
    \bigl(\lambda(t)\,x\bigr)(m)
      = \lambda(t)(m)\cdot x(m) = t^\pair{m,v}\cdot x(m).
  \end{equation*}
  for~$x\in X_\sigma(k)$, $m\in\sigma^\vee\cap M$,
  and the $1$-parameter subgroup~$\lambda$ with differential~$v\in N$.
  In our case, all exponents in
  \begin{equation*}
    \phi_{\alpha}(t)(m)
    % = \bigl(\lambda_{\sigma_p}(t_p)\cdots\lambda_{\sigma_1}(t_1)\,
    %     x_{\sigma_0}\bigr)(m)
    = \prod_i t_i^\pair{m,v_{\sigma_i}}\cdot x_{\sigma_0}(m)
  \end{equation*}
  are non-negative since all $v_{\sigma_i}$ lie in~$\sigma_p$.
  % and $m_j\in\sigma_p^\vee$.
  Hence the map~$\phi_{\alpha}$ % \eqref{definition-A}
  is well-defined and continuous.

  If~$m\in\sigma_p^\perp$, then $m\in\sigma_i^\perp$ for all~$i$,
  hence $\phi_{\alpha}(0)(m)=1$. If $m\notin\sigma_p^\perp$,
  then $\lambda_{\sigma_p}(0)(m)=x_{\sigma_p}(m)=0$, hence
  $\phi_{\alpha}(0)(m)=0$. Therefore, $\phi_{\alpha}(0)=x_{\sigma_p}$.

  Since $\sigma_{p-1}$ is a face of~$\sigma_p$, there is an
  element~$m\in\sigma_p^\vee\cap M$ vanishing on~$\sigma_{p-1}$,
  hence on all~$\sigma_i$, $i<p$, but not on~$\sigma_p$.
  Then
  \begin{equation*}
    \phi_{\alpha}(t)(m)
    = t_p^\pair{m,v_{\sigma_p}}, % (x_{\sigma_0})_j,
  \end{equation*}
  which shows that $\phi_{\alpha}(t)$ determines $t_p$.
\end{proof}

Let $\alpha=(\sigma_0,\dots,\sigma_p)\in\FF(\Sigma)$ be a $p$-simplex.
Applying Lemma~\ref{image-cube} repeatedly % $p$~times
proves that the image~$B(\alpha)$
of $\phi_{\alpha}$ is a $p$-simplex
with vertices~$x_{\sigma_0}$,~\ldots,~$x_{\sigma_p}$.
Its interior is contained in~$\O_{\sigma_0}(\R_+)$,
and its proper faces are the simplices corresponding
to proper subsequences of~$\alpha$. % $(\sigma_0,\dots,\sigma_p)$.
It remains to verify the following claim:

\begin{lemma}
  The interiors of the simplices~$B(\alpha)$,
  $\alpha\in\FF(\Sigma)$,
  form a partition of~$X_\Sigma(\R_+)$.
\end{lemma}

\begin{proof}
  It suffices to show that interiors of the simplices %~$\alpha$
  with initial vertex~$\sigma_0=\sigma$
  partition $\O_\sigma(\R_+)$.
  To this end we define a complete fan~$\Sigma_\sigma$ in~$N(\sigma)_\R$
  which is the
  ``barycentric subdivision''
  of the star of~$\sigma$.
  Its cones~$\tau_\sigma(\alpha)$
  are labelled by simplices~$\alpha=(\sigma_0,\dots,\sigma_p)\in\FF(\Sigma)$
  with initial vertex~$\sigma_0=\sigma$, and
  % $\tau_\sigma(\alpha)$ % \in\Sigma_\sigma$ is
  are spanned by the rays through the
  vectors~$\pi_\sigma(v_{\sigma_1})$,~\ldots,~$\pi_\sigma(v_{\sigma_p})$.
  (Observe that $\pi_\sigma(v_{\sigma_i})$ is an interior point
  of~$\pi_\sigma(\sigma_i)$.)

  The exponential map~$\exp_\sigma\colon N(\sigma)_\R\to\O_\sigma(\R_+)$
  is a real analytic isomorphism, in particular bijective.
  It is clear from Lemma~\ref{image-cube}
  that $\exp_\sigma$ maps the interior of the
  cone~$-\tau_\sigma(\alpha)$ onto the interior
  of~$B(\alpha)$.
  Since the interiors of the cones in~$\Sigma_\sigma$
  partition $N(\sigma)_\R$, this proves the claim.
\end{proof}

\subsection{Case of
  complete~\texorpdfstring{$\Sigma$}{Sigma}
  and arbitrary~\texorpdfstring{$k$}{k}}

\noindent
The inclusion~$\R_+\!\hookrightarrow k$ induces an inclusion
% \begin{equation}\label{inclusion-k}
$X_\Sigma(\R_+)\hookrightarrow X_\Sigma(k)$,
% \end{equation}
% sending the orbit~$\O_\sigma(\R_+)$ to~$\O_\sigma(k)$,
similarly, the norm~$k\to\R_+$, $z\mapsto |z|$
induces a retraction
% \begin{equation}\label{retraction-k}
$X_\Sigma(k)\to X_\Sigma(\R_+)$.
% \end{equation}
Both maps are compatible with the orbit structures.
This implies that the restriction
\begin{equation}\label{surjection}
  \TC(k)\times X_\Sigma(\R_+)\to X_\Sigma(k).
\end{equation}
of the $\TA(k)$-action on~$X_\Sigma(k)$
is surjective and descends to a $\TC(k)$-equivariant bijection
% \begin{equation}
$Y_\Sigma(k)\to X_\Sigma(k)$.
% \end{equation}
This map must be a homeomorphism since $Y_\Sigma(k)$ is compact
and $X_\Sigma(k)$ Hausdorff.

\subsection{Case of arbitrary~\texorpdfstring{$\Sigma$}{Sigma}}
\label{proof-arbitrary}

\noindent
Any rational fan is a subfan of a complete rational fan~$\tilde\Sigma$,
\cf~\cite[Thm.~9.3]{Ewald:96}; equivalently,
any toric variety~$X_\Sigma(k)$ is a $\TA(k)$-stable open subvariety
of a complete toric variety~$X_{\tilde\Sigma}(k)$.
The order complex~$\FF(\tilde\Sigma)$ contains $\FF(\Sigma)$
as a full subcomplex.
% containing
% the simplices of the form~$\alpha=(\sigma_0,\dots,\sigma_p)$ with
% $\sigma_p\in\Sigma$ (and hence $\sigma_i\in\Sigma$ for all~$i$).

Recall from Section~\ref{proof-complete-positive}
that the interior of the simplex~$B(\alpha)$
is contained in the orbit corresponding to the initial vertex of~$\alpha$.
% ~$\O_{\sigma_0}(\R_+)$, where $\sigma_0$
% is the initial vertex of~$\alpha=(\sigma_0,\dots,\sigma_p)$.
Therefore, the closed $\TA(\R_+)$-subvariety
\begin{equation}
  Z = X_{\tilde\Sigma}(\R_+)\setminus X_\Sigma(\R_+)
\end{equation}
is the union of all simplices~$B(\sigma_0,\dots,\sigma_p)$ such
that $\sigma_0$, hence all vertices~$\sigma_i$ are not in~$\Sigma$.
Denote this subcomplex
of~$\FF(\tilde\Sigma)$ by~$L$. Then $\FF(\Sigma)$ and $L$ are full subcomplexes
of~$\FF(\tilde\Sigma)$ on complementary vertex sets.
This implies that $|\FF(\Sigma)|$ is a strong deformation retract
of~$|\FF(\tilde\Sigma)|\setminus|L|$,
\cf~\cite[Lemma~70.1]{Munkres:84}.

The $\TC(k)$-equivariant homeomorphism~%
$Y_{\tilde\Sigma}(k)\to X_{\tilde\Sigma}(k)$
given by~\eqref{surjection} % (with $\tilde\Sigma$ instead of~$\Sigma$)
induces a $\TC(k)$-homeomorphism between
\begin{equation}
  Y = \bigl(\,\TC(k)\times(|\FF(\tilde\Sigma)|\setminus|L|)\,\bigr)
    \bigm/\mathord\sim
\end{equation}
and $X_{\tilde\Sigma}(k)\setminus Z = X_\Sigma(k)$.
The canonical strong deformation retraction~%
$|\FF(\tilde\Sigma)|\setminus|L|\to|\FF(\Sigma)|$
finally induces a $\TC(k)$-equivariant strong deformation
retraction~$Y\to Y_\Sigma(k)$.

\section{Cubical subdivisions}\label{cubical}

\noindent
Any simple polytope~$P$ admits a ``cubical subdivision''
with one full-dimensional cube per vertex,
\cf~\cite[Sec.~4.2]{BuchstaberPanov:02}.
If $\Sigma$ is a complete simplicial fan
(as is the case for the normal fan of a simple polytope),
then the homeomorphism~$Y_\Sigma(\R_+)\approx X_\Sigma(\R_+)$
permits us to define
a cubical structure on~$X_\Sigma(\R_+)$ by setting
\begin{equation}\label{definition-I-sigma}
  I_\tau^\sigma = \bigcup_{\substack{
                      (\sigma_0,\dots,\sigma_p)\in \FF(\Sigma)\\%
                      \tau\le\sigma_0,\;\sigma_p\le\sigma}}
                    B(\sigma_0,\dots,\sigma_p)
\end{equation}
for~$\tau\le\sigma$.
(The right-hand side of~\eqref{definition-I-sigma}
is the standard triangulation of a cube
along the main diagonal, \cf~\cite[Constr.~4.4]{BuchstaberPanov:02},
so $I_\tau^\sigma$ is indeed a cube of dimension~$\dim\sigma-\dim\tau$.)

There is another, more intrinsic description of this subdivision,
which in particular shows
that
it % the subdivision~\eqref{definition-I-sigma}
is canonical
and does not depend on the choice of interior points~$v_\sigma$
used to define the simplices~$B(\alpha)$ in Section~\ref{proof-main-result}.
In fact,
\begin{equation}\label{cube-toric-variety}
  I_0^\sigma=X_\sigma(I)\subset X_\sigma(\R_+),
\end{equation}
where $I$ denotes the multiplicative monoid~$[0,1]$.
To see this, observe that the union of the interiors of the simplices~$B(\alpha)$
with initial vertex~$\sigma_0$ and final vertex~$\sigma_p$
is the image of the interior
of the cone~$-\pi_{\sigma_0}(\sigma_p)\subset N({\sigma_0})_\R$
under the
exponential map~$\exp_{\sigma_0}\colon N(\sigma_0)_\R\to\O_{\sigma_0}(\R_+)$.

Note also that this proof shows that
the canonical inclusion~$X_\Sigma(I)\to X_\Sigma(\R_+)$
is in fact surjective. One sees similarly
that for the disc~$D^2=\{z\in\C:|z|\le 1\}$
% considered as submonoid~$D^2=\{z\in\C:|z|\le 1\}\subset\C$,
one has $X_\Sigma(D^2)=X_\Sigma(\C)$. If $\Sigma$ is regular,
we therefore obtain a canonical decomposition of
the smooth compact toric variety~$X_\Sigma(\C)$
into balls~$(D^2)^n$. % , \cf~\cite[Prop.~22]{Strickland:99}.
If $\Sigma$ is not complete, then it is clear from~\eqref{cube-toric-variety}
that we still have $X_\Sigma(I)=Y_\Sigma(\R_+)$,
and similarly $X_\Sigma(D^2)=Y_\Sigma(\C)$.

If $\Sigma$ is a subfan of the cone spanned by a basis of~$N$, then
$X_\Sigma(\C)$ is the complement of a complex coordinate subspace arrangement,
and $X_\Sigma(D^2)$ is the moment-angle complex
associated with the simplicial fan~$\Sigma$,
considered as a simplicial complex,
see~\cite[Ch.~6]{BuchstaberPanov:02}.
Therefore, Theorem~\ref{main-result} includes
the well-known fact that moment-angle complexes
and complements of complex coordinate subspace arrangements
are equivariantly homotopy-equivalent
\cite[Prop.~20]{Strickland:99}
(see also \cite[Lemma~2.13]{BuchstaberPanov:03}).

\section{Piecewise polynomials}
\label{proof-no-odd-degree}

\noindent
An (integral) piecewise polynomial on the fan~$\Sigma$ is a function~$f$
from the support
\begin{equation}
  |\Sigma|=\bigcup_{\sigma\in\Sigma}\sigma \subset N_\R
\end{equation}
of~$\Sigma$ to~$\Z$ such that for any~$\sigma\in\Sigma$
the restriction~$\at{f}_\sigma$ of~$f$ to~$\sigma$ coincides
with the restriction of some polynomial,
integral with respect to the lattice~$N$.
The set~$\PP(\Sigma;\Z)$ of all piecewise polynomials on~$\Sigma$
is a ring under pointwise addition and multiplication. Moreover,
the canonical identification of~$H^*(B\TC;\Z)$
with the integral polynomials on~$N$ (see Step~1 below)
gives a morphism of rings~$H^*(B\TC;\Z)\to \PP(\Sigma;\Z)$
by restriction of functions to~$|\Sigma|$, hence endows $\PP(\Sigma;\Z)$
with the structure of an $H^*(B\TC;\Z)$-algebra.

The idea of the proof of Theorem~\ref{no-odd-degree}
is to identify the piecewise polynomials
on the fan~$\Sigma$ with the kernel of some ``Mayer--Vietoris differential''
for~$X_\Sigma$ and then to relate this kernel
to the
% Chang--Skjelbred sequence
so-called `Atiyah--Bredon sequence' for~$X_\Sigma$.

\subsection{Step~1}

Recall that any toric variety $X_\Sigma$ is covered
by the affine toric subvarieties~$X_\sigma$ where $\sigma$
runs through $\Sigmamax$, the set of maximal cones.
The intersection of any two affine toric subvarieties
~$X_\sigma$~and~$X_\tau$ is the affine toric subvariety~$X_{\sigma\cap\tau}$.

Any affine toric variety~$X_\sigma$ can be
equivariantly retracted onto its unique closed orbit~$\O_\sigma$
and the latter onto the $\TC$-orbit~$\TC/\TC_\sigma$ of~$x_\sigma$.
This establishes canonical isomorphisms
\begin{equation}\label{equivariant-cohomology-sigma}
  H_\TC^*(X_\sigma;\Z) = H_\TC^*(\O_\sigma;\Z) = H_\TC^*(\TC_\sigma;\Z) = \Z[\sigma],
\end{equation}
where $\Z[\sigma]$ denotes
the polynomials on~$N_\sigma$
(or, equivalently, on~$\sigma\cap N$)
with integer coefficients.
The induced grading on polynomials is twice the usual degree.
Moreover,
% under the canonical isomorphism~$H_\TC^*(X_\sigma;\Z)=\Z[\sigma]$,
for any pair~$\tau\le\sigma$
the map~$H_\TC^*(X_\sigma;\Z)\to H_\TC^*(X_\tau;\Z)$
induced by the inclusion~$X_\tau\hookrightarrow X_\sigma$
corresponds under the isomorphism~\eqref{equivariant-cohomology-sigma}
to the restriction of polynomials from~$N_\sigma$ to~$N_\tau$.
% proof: using a v in the interior of \tau, retract X_\sigma
% onto X_\sigma\cap\bar\O_\tau. This reduces the problem to
% one where \tau = 0, where it is clear.
In the following, we will not distinguish
between polynomials and elements in the various
cohomology groups in~\eqref{equivariant-cohomology-sigma}.

Fix some ordering of~$\Sigmamax$.
A piecewise polynomial on~$\Sigma$ can be given uniquely
by a collection of polynomials $f_\sigma\in\Z[\sigma]$, $\sigma\in\Sigmamax$,
that agree on common intersections.
In other words, we can identify $\PP(\Sigma;\Z)$
% the piecewise polynomials on~$\Sigma$
with the kernel of the map
\begin{equation}\label{mayer-vietoris}
\begin{split}
  \delta\colon
  \bigoplus_{\sigma_0\in\Sigmamax} H_\TC^*(X_{\sigma_0};\Z)
  &\to
  \bigoplus_{\substack{\sigma_0,\sigma_1\in\Sigmamax \\ \sigma_0<\sigma_1}}
        H_\TC^*(X_{\sigma_0\cap\sigma_1};\Z),
  \\
  (\delta f)_{\sigma_0\sigma_1}
    &= \at{f_{\sigma_1}}_{\sigma_0\cap\sigma_1}
     - \at{f_{\sigma_0}}_{\sigma_0\cap\sigma_0}
\end{split}
\end{equation}
where we have used the same notation as in~\cite[\S 8]{BottTu:82}.
(In fact, the map~$\delta$ is the differential %~$d_2^{0,*}$
between the first two columns
of the $E_2$~term of the Mayer--Vietoris spectral sequence
associated to our covering of~$X_\Sigma$ by maximal affine toric subvarieties.)

\subsection{Step~2}

Our first observation is standard (at least for cohomology
with field coefficients).

\begin{lemma}\label{cef-conditions}
  The following conditions are equivalent
  for a toric variety~$X_\Sigma$:
  \begin{enumerate}
  \item\label{criterion-a}
    $H^*(X_\Sigma;\Z)$ is concentrated in even degrees.
  \item\label{criterion-b}
    The Serre spectral sequence
    for the Borel construction of~$X_\Sigma$
    % fibration~$X\hookrightarrow X_\TC\to B\TC$
    degenerates at the~$E_2$~level.
  \item\label{criterion-c}
    The canonical map~$H_\TC^*(X_\Sigma;\Z)\to H^*(X_\Sigma;\Z)$ is a surjection.
  \end{enumerate}
\end{lemma}

\begin{proof}
  The implication~%
  $\hbox{\eqref{criterion-a}}\Rightarrow\hbox{\eqref{criterion-b}}$
  as well as the equivalence~%
  $\hbox{\eqref{criterion-b}}\Leftrightarrow\hbox{\eqref{criterion-c}}$
  hold for any $\TC$-space.
  For~$\hbox{\eqref{criterion-c}}\Rightarrow\hbox{\eqref{criterion-a}}$
  we use that $H_\TC^*(X_\Sigma;\Z)$ injects into~$H_\TC^*(X_\Sigma^\TC;\Z)$,
  see~\cite{FranzPuppe:07}
  or the proof of Proposition~\ref{maximal-full-dim} below.
  Since $X_\Sigma^\TC$ is discrete, this forces $H_\TC^*(X_\Sigma;\Z)$
  to be concentrated in even degrees, hence also $H^*(X_\Sigma;\Z)$.
\end{proof}

\begin{proposition}\label{maximal-full-dim}
  If $X_\Sigma$ satisfies the conditions in Lemma~\ref{cef-conditions},
  then all maximal cones in~$\Sigma$ are full-dimensional.
\end{proposition}

\begin{proof}
We abbreviate $X_\Sigma=X$.

Since in the case of toric varieties all isotropy groups are connected,
the conditions listed in Lemma~\ref{cef-conditions}
imply by a result of Franz--Puppe \cite{FranzPuppe:07}
that the ``Atiyah--Bredon sequence''
\begin{multline}\label{atiyah-bredon}
      0
      \longrightarrow H_\TC^*(X;\Z)
      \stackrel{\iota^*}\longrightarrow H_\TC^*(X_0;\Z)
      \stackrel{\delta^0}\longrightarrow H_\TC^{*+1}(X_1, X_0;\Z)
      \stackrel{\delta^1}\longrightarrow \cdots \\
      \cdots
      \stackrel{\delta^{n-1}}\longrightarrow H_\TC^{*+n}(X_n, X_{n-1};\Z)
      \longrightarrow 0.
\end{multline}
is exact.
(The first part of~\eqref{atiyah-bredon} up to~$H_\TC^{*}(X_1, X_0;\Z)$
is also called the ``Chang--Skjelbred sequence''.)
Here $X_i$ denotes the equivariant $i$-skeleton of~$X$,
\ie, the union of all orbits of dimension at most~$i$.
In particular, $X_0=X^\TC$, the fixed point set.
The map~$\iota^*$ is induced by the inclusion~$\iota\colon X^T\hookrightarrow X$,
and $\delta^i$ is the differential in the
long exact cohomology sequence for the triple~$(X_{i+1},X_i,X_{i-1})$.
Note that while \cite{FranzPuppe:07} work in the setting of
finite $\TC$-CW~complexes, Corollary~\ref{T-CW-complex-C}
allows us to apply this result % the Atiyah--Bredon sequence
to~$X$ even in the non-compact case;
here we use that the canonical $\TC$-homotopy described
in Section~\ref{proof-arbitrary} preserves orbit types.
% \comment{tom Dieck?}

We have
\begin{equation}\label{relative-cohomology-orbits}
  H_\TC^*(X_i,X_{i-1};\Z) = \bigoplus_{\sigma\in\Sigma_{n-i}}
                         H_\TC^*(\bar\O_\sigma,\partial\O_\sigma;\Z),
\end{equation}
and the differential
\begin{equation}
  \delta^i\colon H_\TC^*(X_i,X_{i-1};\Z) \to H_\TC^{*+1}(X_{i+1},X_i;\Z)
\end{equation}
is a ``block matrix'' whose components are the maps
\begin{equation}\label{differential-tau-sigma}
  H_\TC^*(\bar\O_\tau,\partial\O_\tau;\Z) \to
    H_\TC^{*+1}(\bar\O_\sigma,\partial\O_\sigma;\Z)
\end{equation}
for those pairs~$(\sigma,\tau)$ where
$\O_\tau\subset\bar\O_\sigma$, \ie,
where $\tau\in\Sigma_{n-i}$ is a facet of~$\sigma\in\Sigma_{n-i+1}$.

Now assume that $\tau\in\Sigma$ is maximal and of codimension~$k>0$.
Then, by maximality, % $\bar\O_\tau=\O_\tau$, and
$H_\TC^*(\bar\O_\tau,\partial\O_\tau;\Z)=H_\TC^*(\O_\tau;\Z)$ is
a direct summand of~$H_\TC^*(X_k,X_{k-1};\Z)$,
and no non-zero element of~$H_\TC^*(\O_\tau;\Z)$ can be
in the image of the differential~$\delta^{k-1}$.

Let $\sigma$ be a facet of~$\tau$. (If $\tau$ were the zero cone,
then $X$ would a complex torus and $H^1(X;\Z)\neq0$,
contrary to our assumptions.)
The toric variety~$\bar\O_\sigma$
is described by the star of~$\sigma$ in~$\Sigma$
(\cf~\cite[Sec.~3.1]{Fulton:93}),
which we denote by~$\hat\sigma$.
To compute the map~\eqref{differential-tau-sigma},
we replace $\bar\O_\sigma$ by the $\TC$-equivariantly homotopy equivalent
$\TC$-CW~complex~$Y=Y_{\hat\sigma}(\C)\subset\bar\O_\sigma$ and
$\partial\O_\sigma$ by~$Z=Y\cap\partial\O_\sigma$.
Note that $Y\cap\O_\tau$ is a single orbit~$\TC_\tau$
because $\tau$ is maximal.
Let $Y'$ be the space obtained from~$Y$ by replacing this orbit~$\TC_\tau$
by~$\TC_\sigma$, and similarly for~$Z'$.
(This means changing the identification for
the the points above the vertex~$\tau\in\FF(\hat\sigma)$
in~\eqref{topology-complete}.)
From the projection~$Y'\to Y$ we see that the map
\begin{equation}
  H_\TC^*(\bar\O_\tau,\partial\O_\tau;\Z) = H_\TC^*(\TC_\tau;\Z) \to
    H_\TC^{*+1}(Y,Z;\Z) = H_\TC^{*+1}(\bar\O_\sigma,\partial\O_\sigma;\Z)
\end{equation}
factors through
\begin{equation}
  H_\TC^*(\TC_\tau;\Z) \to H_\TC^*(\TC_\sigma;\Z) \to H_\TC^{*+1}(Y',Z';\Z) = H_\TC^{*+1}(Y,Z;\Z).
\end{equation}
The map~$H_\TC^*(\TC_\tau;\Z)\to H_\TC^*(\TC_\sigma;\Z)$ is
the canonical projection~$\Z[\tau]\to\Z[\sigma]$.
Pick a non-zero element~$f_\sigma$ in its kernel.
Then the product of all these~$f_\sigma$, as $\sigma$ runs through
the facets of~$\tau$, is a non-zero element
in the kernel of the differential~$\delta^k$. %
% $H_\TC^*(X_k,X_{k-1};\Z)\supset H_\TC^*(\O_\tau;\Z)\to H_\TC^{*+1}(X_{k+1},X_k;\Z)$.
As it does not lie in the image of~$\delta^{k-1}$,
we get a contradiction to the exactness of the Atiyah--Bredon sequence.
\end{proof}

\subsection{Step~3}

Equations \eqref{equivariant-cohomology-sigma}~and~%
\eqref{relative-cohomology-orbits} (for~$i=0$)  together give
a canonical isomorphism
\begin{equation}
  H_\TC^*(X_0;\Z) = \bigoplus_{\sigma\in\Sigma_n} H_\TC^*(X_\sigma;\Z).
\end{equation}
We finally show that under this isomorphism the kernel of the differential
\begin{equation}
  \delta^0\colon H^*(X_0;\Z)\to H^{*+1}_\TC(X_1, X_0;\Z)
  = \bigoplus_{\sigma\in\Sigma_{n-1}}
      H_\TC^*(\bar\O_\sigma,\partial\O_\sigma;\Z)
\end{equation}
coincides with that of the map~\eqref{mayer-vietoris}.

Since no cone~$\tau\in\Sigma_{n-1}$ is maximal, it is contained
in either one or two full-dimensional cones.
In the first case we have
\begin{equation}
  H_\TC^*(\bar\O_\tau,\partial\O_\tau;\Z) = H_\TC^*(\C, \{0\};\Z) = 0,
\end{equation}
and in the second case
\begin{equation}
  H_\TC^*(\bar\O_\tau,\partial\O_\tau;\Z) = H_\TC^*(\CP^1, \{0,\infty\};\Z)
  \cong \Z[\tau][+1],
\end{equation}
where the last isomorphism is chosen such that
if $\tau$ is the common facet of $\sigma_0$~and~$\sigma_1$,
$\sigma_0<\sigma_1$, then the differential is of the form
\begin{equation}\label{differential-X0-X1}
\begin{split}
  % \Z[\sigma_0]\oplus \Z[\sigma_1] =
  H_\TC^*(\O_{\sigma_0};\Z)\oplus H_\TC^*(\O_{\sigma_1};\Z)
  &\to H_\TC^{*+1}(\bar\O_\tau,\partial\O_\tau;\Z) \\ % = \Z[\tau] \\
  (f_0, f_1) &\mapsto \at{f_1}_{N_\tau} - \at{f_0}_{N_\tau}.
\end{split}
\end{equation}

Consider the following diagram,
where the vertical map on the right sends
each summand~$H_\TC^*(X_{\sigma_0\cap\sigma_1};\Z)=\Z[\tau]$
to~$0$ if $\tau=\sigma_0\cap\sigma_1$ is a not common facet, and
identically onto~$H_\TC^*(\bar\O_\tau,\partial\O_\tau;\Z)=\Z[\tau][+1]$
otherwise. % (Note that this map is surjective.)
\begin{diagram}[LaTeXeqno]
  0 & \rTo & H_\TC^*(X;\Z) & \rTo^{\iota^*}
    & % H_\TC^*(X_0;\Z) =
      \bigoplus_{\sigma\in\Sigma_n} H_\TC^*(\O_\sigma;\Z)
    & \rTo^{\delta^0}
    & % H_\TC^{*+1}(X_1, X_0;\Z) =
      \bigoplus_{\tau\in\Sigma_{n-1}}
        H_\TC^{*+1}(\bar\O_\tau,\partial\O_\tau;\Z) \\
  & & & & \dEq & & \uTo \\
  & & H_\TC^*(X;\Z) & \rTo^{\iota^*}
    & \bigoplus_{\sigma\in\Sigma_n} H_\TC^*(X_\sigma;\Z)
    & \rTo^{\delta}
    & \bigoplus_{\substack{\sigma_0,\sigma_1\in\Sigma_n \\ \sigma_0<\sigma_1}}
        H_\TC^*(X_{\sigma_0\cap\sigma_1};\Z)
\end{diagram}
The commutativity of the right square follows from formulas
\eqref{mayer-vietoris}~and~\eqref{differential-X0-X1}.

Since the differential~$\delta^0$ is the composition
of~$\delta$ and another map, % projection,
the kernel of~$\delta^0$ contains
that of~$\delta$. We know that $\ker\delta^0=H_\TC^*(X;\Z)$.
We also know that the map~$\iota^*$ induced by the inclusion
of the fixed point set is injective, and its image is contained in
the kernel of~$\delta$.
Hence $H_\TC^*(X;\Z)\subset\ker\delta\subset\ker\delta^0=H_\TC^*(X;\Z)$,
so both kernels agree.
This finishes the proof of Theorem~\ref{no-odd-degree}.

\section{Torsion-free cohomology}\label{torsion-free}

\noindent
We now turn our attention to toric varieties
whose equivariant cohomology is not only concentrated
in even degrees, but also torsion-free.
This property can be characterized nicely in terms
of equivariant cohomology; moreover, it behaves
well when passing to orbit closures.

\begin{lemma}\label{equivalence-freeness}
  The ordinary cohomology~%
  $H^*(X_\Sigma;\Z)$ is torsion-free and concentrated in even degrees
  iff the equivariant cohomology~$H_\TC^*(X_\Sigma;\Z)$
  is free over~$H^*(B\TC;\Z)$.
\end{lemma}

\begin{proof}
  % In fact, this holds for any $T$-space~$X$.
  If $\Hodd(X_\Sigma;\Z)$ vanishes, then
  the map~$\iota^*\colon H_\TC^*(X_\Sigma;\Z)\to H^*(X_\Sigma;\Z)$
  is surjective by Lemma~\ref{cef-conditions}. % if $\Hodd(\Sigma;\Z)$ vanishes.
  If moreover $H^*(X_\Sigma;\Z)$ is free over~$\Z$,
  then there exists a section to~$\iota^*$, and
  $H_\TC^*(X_\Sigma;\Z)\cong H^*(X_\Sigma;\Z)\otimes H^*(B\TC;\Z)$
  is free over~$H^*(B\TC;\Z)$ by the Leray--Hirsch Theorem.

  Conversely, if $H_\TC^*(X_\Sigma;\Z)$ is free over~$H^*(B\TC;\Z)$,
  % then $\iota^*$ mu surjective.
  then the Atiyah--Bredon sequence~\eqref{atiyah-bredon} is exact, and
  $H_\TC^*(X_\Sigma;\Z)$ injects
  into~$H_\TC^*(X_\Sigma^T;\Z)=H^*(X_\Sigma^T;\Z)\otimes H^*(B\TC;\Z)$.
  Since $X_\Sigma^T$ is finite,
  this shows that $H_\TC^*(X_\Sigma;\Z)$ is concentrated in even degrees.
  Therefore, $H^*(X_\Sigma;\Z)=H_\TC^*(X_\Sigma;\Z)\otimes_{H^*(B\TC;\Z)}\Z$
  is torsion-free and concentrated in even degrees.
\end{proof}

\begin{proposition}\label{orbit-closure-cef}
  If $H^*(X_\Sigma;\Z)$ is torsion-free and concentrated in even degrees,
  then the same holds true
  for any orbit closure~$\bar\O_\sigma\subset X_\Sigma$.
  %, $\sigma\in\Sigma$
\end{proposition}

Recall that $\bar\O_\sigma$ is again a toric variety,
defined by the star of~$\sigma$ in~$\Sigma$.

\begin{proof}
  Note first that it is enough to prove
  that $H^*(\bar\O_\sigma;\F_p)$ is concentrated in even degrees
  for all primes~$p$.
  % the statement for all prime fields~$k=\F_p$ instead of~$\Z$.
  Moreover,
  since $\bar\O_\sigma$ is a component of~$X_\Sigma^{\TC_\sigma}$,
  it suffices to consider fixed point sets~$X_\Sigma^G$,
  where $G\subset \TC$ is any subtorus.

  We use that for a (sufficiently ``nice'') $\TC$-space $X$ one has
  \begin{equation}
    \dim H^*(X;\F_p) \ge \dim H^*(X^\TC;\F_p)
  \end{equation}
  with equality iff
  % $X$ is cef,
  $H_\TC^*(X;\F_p)$ is free over~$H^*(B\TC;\F_p)$,
  \cf~\cite[Cor.~3.1.14 \&~3.1.15]{AlldayPuppe:93}.
  (There rational coefficients are used. However, a look at the proof shows
  that in the case of connected isotropy groups coefficients can be taken
  in any field.)
  
  Set $X=X_\Sigma$ and $Y=X_\Sigma^G$.
  In this case we have
  \begin{equation}
    \dim H^*(X;\F_p) \ge \dim H^*(Y;\F_p) \ge \dim H^*(Y^\TC;\F_p).
  \end{equation}
  Since $Y^\TC=X^\TC$, all inequalities must be equalities.
  Therefore $H_\TC^*(Y^G;\F_p)$ surjects onto~$H^*(Y^G;\F_p)$,
  so the latter is concentrated in even degrees.
\end{proof}

\begin{question}
  Is the property ``$\Hodd(X_\Sigma;\Z)=0$'' inherited by orbit closures
  even in the presence of torsion?
\end{question}

In the course of the proof of Lemma~\ref{equivalence-freeness}
we showed that if $\Hodd(X_\Sigma;\Z)$
%the odd-dimensional cohomology of~$X_\Sigma$
vanishes,
then $H_\TC^*(X_\Sigma;\Z)$
% the equivariant cohomology of~$X_\Sigma$
injects into the free $H^*(B\TC;\Z)$-module $H_\TC^*(X_\Sigma^\TC;\Z)$,
so $H_\TC^*(X_\Sigma;\Z)$ cannot have $\Z$-torsion.
For a general $\TC$-space~$X$,
this last property together with the degeneration
of the Serre spectral sequence
does \emph{not} guarantee
that
$H^*(X;\Z)$
% the ordinary cohomology
itself is torsion-free.
(See \cite[Ex.~5.2]{FranzPuppe:07} for a counterexample.)
But Proposition~\ref{no-torsion-smooth-compact} asserts
that this conclusion is valid for toric varieties which are smooth or compact.

\begin{proof}[Proof of Proposition~\ref{no-torsion-smooth-compact}]
  We again write $X=X_\Sigma$.
  Assume first that $X$ is compact.
  Then all terms~$H_\TC^{*+i}(X_i, X_{i-1};\Z)$
  in the Atiyah--Bredon sequence~\eqref{atiyah-bredon}
  are free over~$\Z$: in fact, equation~\eqref{relative-cohomology-orbits}
  becomes
  \begin{equation}\label{relative-cohomology-orbits-2}
  \begin{split}
    H_\TC^*(X_i,X_{i-1};\Z)
    &= \bigoplus_{\sigma\in\Sigma_{n-i}} H_\TC^*(\bar\O_\sigma,\partial\O_\sigma;\Z)\\
    &= \bigoplus_{\sigma\in\Sigma_{n-i}} H_\TC^*(\O_\sigma;\Z)[+i]
     \cong \bigoplus_{\sigma\in\Sigma_{n-i}} \Z[\sigma][+i].
  \end{split}
  \end{equation}
  (This is the $E_1$~term of the spectral sequence
  considered in~\cite{Fischli:92}.)
  $H_\TC^*(X;\Z)$, being a submodule of~$H_\TC^*(X_0;\Z)$, is free over~$\Z$ as well.
  Hence, the Atiyah--Bredon sequence over a finite field~$\F_p$
  is obtained by tensoring the integral version~\eqref{atiyah-bredon}
  with~$\F_p$, and this does not affect exactness.

  It actually holds true for any field~$k$ as coefficients that the
  exactness of the Atiyah--Bredon sequence implies
  the freeness of~$H_\TC^*(X;k)$ over~$H^*(B\TC;k)$.
  (This can be seen by inspecting the proofs in~%
  \cite[Sec.~4.8]{BarthelBrasseletFieselerKaup:02}
  or~\cite[Sec.~4]{FranzPuppe:07}.)
  Since $H_\TC^*(X;k)$ % =H^*(X;k)\otimes H^*(B\TC;k)$
  injects into~$H_\TC^*(X_0;k)$ and the latter module
  is concentrated in even degrees, the same applies to the former,
  hence also to its quotient~$H^*(X;k)$.
  By considering the Universal Coefficient Theorem
  for prime fields~$k=\F_p$, one sees that
  this is impossible if $H^{\rm even}(X;\Z)$ has torsion.

  Suppose now that $X$ is smooth, and let $\tilde X$ be a toric
  compactification of~$X$, \cf~Section~\ref{proof-arbitrary}.
  Set $Z=\tilde X\setminus X$. By Lefschetz duality
  (\cf~\cite[Thm.~70.2]{Munkres:84}),
  $H^*(\tilde X,Z;\Z)=H_{2n-*}(X;\Z)$ is also concentrated in even degrees,
  and the reasoning for the compact case carries over to
  the pair~$(\tilde X,Z)$ instead of~$X$. Hence,
  $H^*(\tilde X,Z;\Z)$ is torsion-free, and therefore $H^*(X;\Z)$
  as well.  
\end{proof}

Assume that $H^*(X_\Sigma;\R)$ is concentrated in even degrees.
Then a reasoning analogous to that in Section~\ref{proof-no-odd-degree}
shows that $H_\TC^*(X_\Sigma;\R)=\PP(X_\Sigma;\R)$ is free
over the polynomial ring~$H^*(B\TC;\R)$.
A result of Yuzvinsky \cite[Cor.~3.6]{Yuzvinsky:92} implies that the reduced
homology of all links in~$\Sigma$ vanishes except in top degrees.
(In the case of a simplicial fan,
this is Reisner's Cohen--Macaulay criterion,
\cf~\cite[Cor.~5.3.9]{BrunsHerzog:98}.)
For compact toric varieties, we can give a short topological proof
of this fact, even with integer coefficients.

\begin{proposition}\label{free-link}
  If $X_\Sigma$ is compact and
  $H^*(X_\Sigma;\Z)$ concentrated in even degrees,
  then $\tilde H_i(\lk\sigma;\Z)=0$ for all $\sigma\in\Sigma$
  and all~$i<n-\dim\sigma-1$.
\end{proposition}

\begin{proof}
  By Propositions \ref{no-torsion-smooth-compact}~and~\ref{orbit-closure-cef},
  it is enough to consider the case~$\sigma=0$.

  Because $X=X_\Sigma$ is compact, we have
  \begin{equation} %\label{relative-cohomology-orbits-2}
    H_\TC^*(X_i,X_{i-1};\Z)
     \cong \bigoplus_{\sigma\in\Sigma_{n-i}} \Z[\sigma][+i],
  \end{equation}
  where the isomorphism is determined by a choice of orientations
  of the cones. Hence, the part
  \begin{equation} %\label{atiyah-bredon}
      % 0
      % \longrightarrow H_\TC^*(X;\Z)
      % \stackrel{\iota^*}\longrightarrow
    H_\TC^0(X_0;\Z)
      \stackrel{\delta^0}\longrightarrow H_\TC^{1}(X_1, X_0;\Z)
      \stackrel{\delta^1}\longrightarrow % \cdots \\
      \cdots
      \stackrel{\delta^{n-1}}\longrightarrow H_\TC^{n}(X_n, X_{n-1};\Z)
      \longrightarrow 0.
  \end{equation}
  of the Atiyah--Bredon sequence computes the homology
  with closed support of~$|\Sigma|$. This is, up to a degree shift by~$1$,
  the homology of link of the zero cone.
  Since $H^*(X;\Z)$ is concentrated in even degrees,
  this sequence is exact,
  which means $\tilde H_i(\lk 0;\Z)=0$ for all~$i<n-1$.
\end{proof}

We conclude with a remark about hereditary fans.
A fan~$\Sigma$ is called \emph{hereditary}
if all maximal cones are full-dimensional
and if for every~$\tau\in\Sigma$ one has that
\begin{equation}\label{property-hereditary}
  \advance\textwidth by -1in
  \left\{\;
  \parbox[c]{\textwidth}{%
    any two maximal cones~$\sigma$,~$\sigma'$ in the star of~$\tau$
    can be joined
    by a sequence~$\sigma=\sigma_0$,~\ldots,~$\sigma_k=\sigma'$ of
    maximal cones in the star of~$\tau$ such that
    $\sigma_{i-1}$ and $\sigma_i$ have a common facet, $1\le i\le k$;}
  \right.
\end{equation}
see \cite{BilleraRose:92}.

\begin{proposition}\label{no-odd-hereditary}
  If $H^*(X_\Sigma;\Z)$ is concentrated in even degrees,
  then $\Sigma$ is hereditary.
\end{proposition}

\begin{proof}
  We know from Proposition~\ref{maximal-full-dim}
  that all maximal cones in~$\Sigma$ are full-di\-men\-sional.
  Group these cones in~$\Sigma$ into
  ``connected components'' in the sense that all cones
  in a component can be connected by full-dimensional cones
  sharing a common facet. Then the number of these components
  is the dimension of the free $\Z$-module of ``piecewise constant functions''
  on~$\Sigma$, in other words, the dimension of the kernel
  of the differential~$\delta^0$ in the
  Atiyah--Bredon sequence. But this equals $\dim H_\TC^0(X_\Sigma;\Z)=1$
  because the sequence is exact. So condition~\eqref{property-hereditary}
  holds for~$\tau=0\in\Sigma$, the zero cone.

  To reduce the case of general~$\tau\in\Sigma$ to
  the case~$\tau=0$, we consider
  the orbit closure~$\bar\O_\tau\subset X_\Sigma$.
  By Proposition~\ref{orbit-closure-cef},
  % $\bar\O_\tau$ is cef over~$\Q$, which is to say that
  $H^*(\bar\O_\tau;\Z)$ is torsion-free and concentrated in even degrees.
  Therefore, condition~\eqref{property-hereditary} holds
  for the zero cone in the star of any~$\tau$, which means that
  it holds for all~$\tau\in\Sigma$.
\end{proof}

We leave it to the reader to check that
it would actually be enough to assume $\Hodd(X_\Sigma;\Q)=0$.
But even over the rationals one cannot hope for a converse
to Proposition~\ref{no-odd-hereditary}:
An example originally due to Eikelberg and further studied
by Barthel--Brasselet--Fieseler--Kaup
\cite[Ex.~3.5]{BarthelBrasseletFieselerKaup:96}
shows that two combinatorially equivalent complete fans~$\Sigma$ and~$\Sigma'$
in~$\R^3$
can lead to toric varieties $X_\Sigma$ and~$X_{\Sigma'}$ with
$\Hodd(X_\Sigma;\Q)=0$ and $\Hodd(X_{\Sigma'};\Q)=H^3(X_{\Sigma'};\Q)\cong\Q$.

\end{document}